\documentclass[12pt]{article}
\usepackage{amsthm}
\usepackage{amssymb}
\usepackage{amsmath}
\usepackage{enumerate}
\usepackage[mathscr]{euscript}

\usepackage[top=1.1 in, bottom=1.3 in, left=1 in, right=1 in]{geometry} 
\usepackage{graphicx}
\newtheorem{theorem}{Theorem}[section]
\newtheorem{definition}{Definition}[section]
\newtheorem{prop}{Proposition}[section]

\newtheorem{lemma}{Lemma}[section]

\theoremstyle{remark}
\newtheorem{example}{ Example}[section]
\newtheorem{remark}{Remark}

\usepackage{tikz}
\usetikzlibrary{arrows,chains,matrix,positioning,scopes}
\makeatletter
\tikzset{join/.code=\tikzset{after node path={%
			\ifx\tikzchainprevious\pgfutil@empty\else(\tikzchainprevious)%
			edge[every join]#1(\tikzchaincurrent)\fi}}}
\makeatother
\tikzset{>=stealth',every on chain/.append style={join},
	every join/.style={->}}
\tikzstyle{labeled}=[execute at begin node=$\scriptstyle,
execute at end node=$]
\date{}
\begin{document}
	\title{On projective representations of Plesken Lie algebras}
	\author{P G Romeo and Arjun S N\\Dept. of Mathematics\\ Cochin University of Science and Technology\\
	Kerala, INDIA.}
	\maketitle
	\begin{abstract}
		In this article we describe the projective representation of Plesken Lie algebras and equivalent central extensions of the same. Further it is also shown that there exists a
		bijective correspondence between second cohomology group, equivalent central extensions and projectively equivalent projective representations of Plesken Lie algebras.
	\end{abstract}
	
	\section{Introduction}
	Lie algebras were introduced by Sophus Lie in connection with his studies of Lie groups; Lie groups are not only groups but also smooth manifolds, the group operations being smooth. The Lie algebras are certain vector spaces endowed with an operation called the Lie  bracket satisfying certain algebraic properties. Further there is a fruitful Lie group Lie algebra correspondence which allows one to relate a Lie group to a Lie algebra or vice versa. The Lie algebra, being a linear object, is more immediately accessible than the group. As with any associative algebra, the group algebra $\mathbb{F}[G]$ for any finite group $G$ over the field $\mathbb{F}$ can be made into a Lie algebra by means for a bracket product.  In \cite{Plesken}, Arjeh M. Cohen and D. E. Taylor introduced certain Lie algebra of a finite group and call it as Plesken Lie algebra. In the paper they discussed the structure of Plesken Lie algebras and explicitly determine the groups for which the Plesken Lie algebra simple and semisimple over complex field. In \cite{Pleskenfinite}, John Cullinan and Mona Merling describes the structural properties of Plesken Lie algebra over finite fields. In \cite{romeo}, Arjun and Romeo describes the linear representations of Plesken Lie algebras which are induced from the group representations. Also they determine the irreducible Plesken Lie algebra representations.

	\par  In group theory, Schur introduced projective representation to determine all the finite groups contained in $GL(n, \mathbb{C})$. The projective representation of a finite group is a homomorphism from that group into the projective linear group (see cf. \cite{Karpilovsky}). These representations appear naturally in the study of ordinary representation of groups and are known to have many applications in other areas of Physics and Mathematics. 
	By definition, every linear representation of a group is projective, but the converse is not true. In the early 1900's, multipliers and covers were first studied by I. Schur in finite group theory. The multiplier of a group is now defined to be the second cohomology group $H^2(G, \mathbb{C}^*)$. In group theory, the multiplier of a group is unique up to isomorphism, but the corresponding cover is not necessarily unique. The notions of multilpliers and covers of Lie algebras were first introduced by P. G. Batten  in his Ph.D Thesis \cite{Batten}. Batten proved that the multilplier for a finite dimensional Lie algebra $L$ is isomorphic to $H^2(L, \mathbb{C})$, the second cohomology group of $L$. Also he showed that covers of Lie algebras are not isomorphic.
	
	\par In this article, our goal is to describe the projective representation of Lie algebras. The theory of projective representations involves understanding homomorphisms from a Lie algebra to the quotient Lie algebra $\mathfrak{gl}(V)/\{kI_V: k \in \mathbb{C} \}$.  Our aim is to explain the multipliers and covers of Lie algebras Plesken Lie algebras using projective representations. Here we explicitly showed the correspondence between the Schur multiplier $H^2(\mathcal{L}(G), \mathbb{C})$, equivalent central extensions and projectively equivalent projective representations of the Plesken Lie algebra $\mathcal{L}(G)$. 
	
	\par  Throughout this paper $\mathcal{L}(G)$ denote the Plesken Lie algebra of a finite group $G$ over the field of complex numbers.
	
	\section{Central extensions}
	In this section we describe the relations between a central extension of a Plesken Lie algebra $\mathcal{L}(G)$ by an abelian Plesken Lie algebra $\mathcal{L}(H)$ and $H^2(\mathcal{L}(G), \mathcal{L}(H))$ and it is shown that, up to equivalence of extensions, central extensions and second homology groups are essentially the same.
	\begin{definition} 
		$G$ and $H$ are two finite groups such that $\mathcal{L}(G)$ and $\mathcal{L}(H)$ are the corresponding Plesken Lie algebras of which $\mathcal{L}(H)$ is abelian. Then $(e; f, g)$ is an extension of $\mathcal{L}(G)$  by $\mathcal{L}(H)$	if there exists a Lie algebra $e$ such that the following is a short exact sequence : 
		\begin{center}
			$0 \to \mathcal{L}(H)\xrightarrow{f} e \xrightarrow{g} \mathcal{L}(G) \to 0$
		\end{center}
	\end{definition}
	An extension $(e;f,g)$ is central if $f(\mathcal{L}(H)) \subseteq Z(\mathcal{L}(G))$ where $Z(\mathcal{L}(G))$ is the center of $\mathcal{L}(G)$.
	\begin{example}\label{central} Consider the subgroups $G=\{ \begin{pmatrix}1&0&b\\0&1&c\\0&0&1\end{pmatrix} : b, c\in \mathbb{Z}/ p \mathbb{Z}\}$ and $H=\{ \begin{pmatrix}1&0&b\\0&1&0\\0&0&1\end{pmatrix} : b \in \mathbb{Z}/ p \mathbb{Z}\}$ of the Heisenberg group $H(\mathbb{Z}/ p \mathbb{Z})$. 
		Then the Plesken Lie algebras of $G$ and $H$ are	\begin{equation*}
			\begin{split}
				\mathcal{L}(G)&= span_{\mathbb{C}}\{ \begin{pmatrix}0&0&b\\0&0&c\\0&0&0\end{pmatrix} : b, c\in \mathbb{Z}/ p \mathbb{Z}\}
				\text{ and } \\
				\mathcal{L}(H)&= span_{\mathbb{C}}\{ \begin{pmatrix}0&0&b\\0&0&0\\0&0&0\end{pmatrix} : b\in \mathbb{Z}/ p \mathbb{Z}\}
			\end{split}
		\end{equation*}
		respectively. Consider 
		\begin{center}
			$e= span_{\mathbb{C}}\{ \begin{pmatrix}0&a&b\\0&0&c\\0&0&0\end{pmatrix} : a, b, c\in \mathbb{Z}/ p \mathbb{Z}\}$,
		\end{center}
		$f : \mathcal{L}(H)\to e$  is the identity inclusion and $g : e \to \mathcal{L}(G)$ is given by
		\begin{center}
			$g(\begin{pmatrix}0&a&b\\0&0&c\\0&0&0\end{pmatrix}) = \begin{pmatrix}0&0&a\\0&0&c\\0&0&0\end{pmatrix}$.
		\end{center}
		Then $f$ and $g$ are Lie algebra homomorphisms and $ker(g) = Im(f)$. Also $f(\mathcal{L}(H)) = \mathcal{L}(H) = Z(e)$, the center of $e$. Thus  $0 \to \mathcal{L}(H)\xrightarrow{f} e \xrightarrow{g} \mathcal{L}(G) \to 0$ is a central extension.

		%(since $\mathcal{L}(G)\cong e/\mathcal{L}(H)$, $g$ is a homomorphism).

	\end{example}
	\subsection{Equivalent central extensions}
	Two central extensions
	\begin{center}
		$0\to \mathcal{L}(H) \xrightarrow{f} e \xrightarrow{g} \mathcal{L}(G) \to 0$
	\end{center}
	and
	\begin{center}
		$0\to \mathcal{L}(H) \xrightarrow{f'} e' \xrightarrow{g'} \mathcal{L}(G) \to 0$
	\end{center}
	of $\mathcal{L}(G)$ by $\mathcal{L}(H)$ are equivalent if there exists a Lie algebra homomorphism $\phi : e\to e'$ such that the following diagram commutes:
	\begin{equation}
		\begin{tikzpicture}
			\matrix (m) [matrix of math nodes, row sep=3em, column sep=3em]
			{ 0 & \mathcal{L}(H) & e  & \mathcal{L}(G) & 0 \\
				0 & \mathcal{L}(H) & e'  & \mathcal{L}(G) & 0 \\ };
			{ [start chain] \chainin (m-1-1);
				\chainin (m-1-2);
				{ [start branch=A] \chainin (m-2-2)
					[join={node[right,labeled] {id}}];}
				\chainin (m-1-3) [join={node[above,labeled] {f}}];
				{ [start branch=B] \chainin (m-2-3)
					[join={node[right,labeled] {\phi}}];}
				\chainin (m-1-4) [join={node[above,labeled] {g}}];
				{ [start branch=C] \chainin (m-2-4)
					[join={node[right,labeled] {id}}];}
				\chainin (m-1-5); }
			{ [start chain] \chainin (m-2-1);
				\chainin (m-2-2);
				\chainin (m-2-3) [join={node[above,labeled] {f'}}];
				\chainin (m-2-4) [join={node[above,labeled] {g'}}];
				\chainin (m-2-5); }
		\end{tikzpicture}
	\end{equation}
	\begin{remark}
		If such a homomorphism $\phi$ exists, then it must be an isomorphism ( see \cite{Batten}).
	\end{remark}
Now recall the  second cohomology group of a Lie algebra (cf.\cite{Batten}). Let $\mathcal{L}(G)$ be a Plesken Lie algebra of $G$ over $\mathbb{C}$. The set of 2-cocycles is given by
\begin{center}
	$Z^2(\mathcal{L}(G), \mathbb{C}) = \{ f: \mathcal{L}(G) \times  \mathcal{L}(G) \to \mathbb{C} : f \text{ is bilinear and } f([x, y], z)$\\\ \hspace{1in} $+f([y, z], x)+f([z,x],y) = 0  \ \forall x, y, z \in \mathcal{L}(G)\}$
\end{center}
and the set of 2-coboundaries are 
\begin{center}
	$B^2(\mathcal{L}(G), \mathbb{C})=\{ f: \mathcal{L}(G) \times  \mathcal{L}(G) \to \mathbb{C} : f \text{ is bilinear and there exists  } $\\ \hspace{1in} $\sigma : \mathcal{L}(G) \to \mathbb{C} \text{ such that } f(x, y)= -\sigma([x, y]) \}$.
\end{center}
Then the second cohomology group of $\mathcal{L}(G)$ is given by
\begin{center}
	$H^2(\mathcal{L}(G), \mathbb{C}) = \frac{Z^2(\mathcal{L}(G), \mathbb{C})}{B^2(\mathcal{L}(G), \mathbb{C})}.$
\end{center}
Note that two 2-cocycles $\alpha_1$ and $\alpha_2$ are said to be cohomologous ( that is, they have the same cohomology class) if there exists a linear map $\sigma : \mathcal{L}(G) \to  \mathbb{C}$ such that
\begin{center}
	$\alpha_2(x, y) - \alpha_1(x, y) = -\sigma([x, y])$
\end{center}

\begin{lemma}
	Let $\alpha \in Z^2(\mathcal{L}(G), \mathbb{C})$. Then $\alpha(x, x) = \alpha(x,0)= \alpha(0,x)= 0$
	
\end{lemma}
\begin{proof}
	Suppose $\alpha \in Z^2(\mathcal{L}(G), \mathbb{C})$. Then for any $x, y, z \in \mathcal{L}(G)$,
	\begin{equation}
		\alpha([x, y],z)+\alpha([y,z],x)+\alpha([z,x],y) = 0
	\end{equation}
	
	Take $x=y= z$, then $(2)$ implies that $3\alpha(0,x) = 0$. That is, $\alpha(0,x)=0$. Then alternating  property of $\alpha$ gives that $\alpha(x, 0)=0$ and $\alpha(x,x)=0$.
\end{proof}
The following theorem gives us the relation between the schur multiplier $H^2(\mathcal{L}(G), \mathbb{C})$ and central extensions of $\mathcal{L}(G)$ by $\mathbb{C}$.
\begin{theorem}
	Let $\mathcal{L}(G)$ be a Plesken Lie algebra of a finite group $G$. Then there is a bijective correspondence between $H^2(\mathcal{L}(G), \mathbb{C})$ and equivalence class of central extensions of $\mathcal{L}(G)$ by $\mathbb{C}$.
\end{theorem}

\begin{proof}
	Let $X$ denotes the equivalence classes of central extensions of $\mathcal{L}(G)$ by $\mathbb{C}$. Define $\Psi: X \to H^2(\mathcal{L}(G), \mathbb{C})$ by
	\begin{center}
		$\Psi([e]) = [\alpha]$
	\end{center}
	For $(e;f,g)$ be a central extension of $\mathcal{L}(G)$ by $\mathbb{C}$ and $s$ be a section map of $g$ (a linear map $s:e \to \mathcal{L}(G)$ such that $g \circ s = I_{\mathcal{L}(G)}$). Define $\alpha : \mathcal{L}(G) \times \mathcal{L}(G)\to \mathbb{C}$ by
	\begin{center}
		$\alpha(x, y) = [s(x), s(y)]- s([x, y]).$
	\end{center}
	Then $\alpha(x, y) \in \mathbb{C}$ and for any $x, y, z \in \mathcal{L}(G)$ and $c \in \mathbb{C}$
	\begin{equation}
		\begin{split}
			\alpha(cx+y,z)&= [s(cx+y), s(z)] - s([cx+y, z])\\
			&= [s(cx)+s(y), s(z)] - s([cx, z]+ [y, z])\\
			&=[cs(x), s(z)] + [s(y), s(z)] -  cs([x, z]) -s([y, z])\\
			&=c\alpha(x, z)+ \alpha(y, z)
		\end{split}
	\end{equation}
	and $\alpha(x, cy+z) = c\alpha(x, y)+ \alpha(x, z)$, hence $\alpha$ is bilinear. Also
	\begin{equation}
		\begin{split}
			\alpha([x,y],z)+&\alpha([y, z], x)+\alpha([z, x], y)\\
			& = [s([x, y]), s(z)] - s([[x,y], z]) +[s([y, z],s(x)] \\
			& \quad -s([[y, z], x]) + [s([z, x]), s(y)] - s([[z,x], y])\\
			&= [[s(x), s(y)] - \alpha(x, y), s(z)] + [[s(y), s(z)] - \alpha(y, z), s(x)]\\
			&\quad  [[s(z), s(x)] - \alpha(z, x), s(y)]
			-s([[x,y], z]+ [[y, z], x]+ [[z,x], y]) \\
			&=0
		\end{split}
	\end{equation}
	(using Jacobi identity in $e$ and $\mathcal{L}(G)$ and also using the fact that the extension is central). Thus $\alpha \in Z^2(\mathcal{L}(G), \mathbb{C})$. To see that $\alpha$ is independant of choice of section, consider a section $s': \mathcal{L}(G) \to e$ of $g$ and $\beta$ the corresponding 2-cocycle, then
	\begin{equation}
		\begin{split}
			g(s'(x) - s(x)) = 0 \quad &\Rightarrow s'(x)- s(x) \in Ker(g) \\
			&\Rightarrow s'(x)- s(x) \in Im(f) \cong \mathbb{C}\\
			&\Rightarrow s'(x) - s(x) = c_x \text{ for some } c_x \in \mathbb{C}
		\end{split}
	\end{equation}
	
	Define $\sigma : \mathcal{L}(G) \to e$ by
	\begin{center}
		$\sigma(x) = c_x$
	\end{center}
	then $\sigma$ is linear and
	
	\begin{equation}
		\begin{split}
			\beta(x,y) &= [s'(x), s'(y)] - s'([x, y]) \\
			&=[c_x+ s(x), c_y+s(y)] - (c_{[x, y]} + s([x, y])) \\
			&= [c_x, c_y] + [c_x, s(y)] +[s(x), c_y] \\
			&\hspace{0.75in} +[s(x), s(y)] - s([x, y]) - c_{[x, y]} \\
			&= \alpha(x, y) - \sigma([x, y])
		\end{split}
	\end{equation}
	thus $\alpha$ and $\beta$ are cohomologous. i.e., $[\alpha] = [\beta]$. ie., $\alpha$ does not depend on the choice of the section.
	\par Consider two equivalent central extensions $(e; f, g)$ and $(e'; f',g')$ of $\mathcal{L}(G)$ by $\mathbb{C}$. Then there is a Lie algebra homomorphism $\phi : e \to e'$ such that the following diagram commutes.
	\begin{equation}
		\begin{tikzpicture}
			\matrix (m) [matrix of math nodes, row sep=3em, column sep=3em]
			{ 0 & \mathbb{C} & e  & \mathcal{L}(G) & 0 \\
				0 & \mathbb{C} & e'  & \mathcal{L}(G) & 0 \\ };
			{ [start chain] \chainin (m-1-1);
				\chainin (m-1-2);
				{ [start branch=A] \chainin (m-2-2)
					[join={node[right,labeled] {id}}];}
				\chainin (m-1-3) [join={node[above,labeled] {f}}];
				{ [start branch=B] \chainin (m-2-3)
					[join={node[right,labeled] {\phi}}];}
				\chainin (m-1-4) [join={node[above,labeled] {g}}];
				{ [start branch=C] \chainin (m-2-4)
					[join={node[right,labeled] {id}}];}
				\chainin (m-1-5); }
			{ [start chain] \chainin (m-2-1);
				\chainin (m-2-2);
				\chainin (m-2-3) [join={node[above,labeled] {f'}}];
				\chainin (m-2-4) [join={node[above,labeled] {g'}}];
				\chainin (m-2-5); }
		\end{tikzpicture}
	\end{equation}
	
	ie., $\phi \circ f = f'$ and $g = g' \circ \phi$. Let $s: \mathcal{L}(G) \to e$ be a  section of $g$. Then there is a 2-cocycle $\alpha$  such that
	\begin{center}
		$\alpha(x, y) = [s(x), s(y)]  - s([x,y])$
	\end{center}
	Define $s' = \phi \circ s$, then $s' : \mathcal{L}(G) \to e'$ is a linear map such that
	\begin{center}
		$g'\circ s' = g'\circ \phi \circ s = g \circ s = I_{\mathcal{L}(G)}$.
	\end{center}
	ie., $s'$ is a section of $g'$ and there is a 2-cocycle $\beta \in Z^2(\mathcal{L}(G) , \mathbb{C})$ such that
	\begin{equation}
		\begin{split}
			\beta(x, y)&= [s'(x), s'(y)] - s'([x, y]) \\
			&= [(\phi\circ s)(x), (\phi\circ s)(y)] - (\phi\circ s)([x,y])\\
			&= \phi([s(x), s(y)] - s([x, y])) \\
			&= \phi(\alpha(x, y)) \\
			&= \alpha(x, y) \text{ \ as $\alpha(x,y) \in \mathbb{C}$}
		\end{split}
	\end{equation}
	hence two equivalent central extensions maps same element in $H^2(\mathcal{L}(G), \mathbb{C})$ via $\phi$ and thus $\Psi$ is well defined.
	\par Let $(e; f, g)$ and $(e'; f', g')$ are two central extensions of $\mathcal{L}(G)$ by $\mathbb{C}$ such that $\Psi([e]) = [\alpha]$ and $\Psi([e'])= [\beta]$. Suppose that $\Psi([e]) = \Psi([e'])$. That is, $[\alpha] = [\beta]$. Let $s$ and $s'$ be the sections of $g$ and $g'$ respectively. Then we have
	\begin{center}
		$\alpha(x, y) = [s(x), s(y)]  - s([x,y])$ and $\beta(x, y) = [s'(x), s'(y)]  - s'([x,y])$
	\end{center}
	Since $\alpha$ and $\beta $ are cohomologous, there is a linear map $\sigma : \mathcal{L}(G) \to \mathbb{C}$ such that
	\begin{center}
		$\alpha(x, y)- \beta(x, y)= -\sigma([x, y])$
	\end{center}
	Note that every element $x$ in $e$ can be uniquely written as $x= c_y + s(y)$ for some $c_y \in \mathbb{C}$ (since $g$ is a projection of $e$, $e = Ker(g)\oplus Im(g)$). %\cite{James}).
	To prove $e$ and $e'$ are equivalent, we have to find a Lie algebra homomorphism $\phi : e \to e'$ such that the diagram (7) commutes.

	For define $\phi : e\to e'$ by
	\begin{center}
		$\phi(c_y+s(y))=c_y+s'(y)+\sigma(y)$
	\end{center}
	Then for any $c_y+s(y),  c_{y'}+s(y')\in e$ and $k \in\mathbb{C}$,
	\begin{equation}
		\begin{split}
			\phi(k(c_y+s(y))+ (c_{y'}+s(y')))&= \phi(kc_y+c_{y'}+ s(ky+ y'))\\
			&= kc_y+c_{y'}+ s'(ky+ y')+ \sigma(ky + y')\\
			&= kc_y+c_{y'}+ ks'(y)+ s'(y')+ k\sigma(y)+ \sigma(y')\\
			&= k\phi(c_y+s(y))+ \phi(c_{y'}+s(y'))
		\end{split}
	\end{equation}
	Thus $\phi$ is linear. Also
	\begin{equation}
		\begin{split}
			[\phi(c_y+s(y)), \phi(c_{y'}+s(y')]&= [c_y+s'(y)+\sigma(y), c_{y'}+s(y')+\sigma(y')]\\
			&= [c_y, c_{y'}] + [c_y, s'(y')]+ [c_y,\sigma(y')]\\
			&\quad +[s'(y), c_{y'}] + [s'(y), s'(y')]+ [s'(y),\sigma(y')]\\
			&\quad +[\sigma(y), c_{y'}] + [\sigma(y), s'(y')]+ [\sigma(y),\sigma(y')]\\
			&= [s'(y), s'(y')]
		\end{split}
	\end{equation}
	(since $\sigma:\mathcal{L}(G) \to \mathbb{C}$ and $\mathbb{C} \cong f(\mathbb{C}) \subseteq Z(e)$, $[\sigma(y), s(y')] = 0$ for all $s(y') \in e$) and
	\begin{equation}
		\begin{split}
			\phi([c_y+s(y), c_{y'}+s(y')])&= \phi([c_y, c_{y'}] + [c_y, s(y')]+[s(y), c_{y'}] + [s(y), s(y')]\\
			&= \phi([s(y), s(y')])\\
			&=\phi(\alpha(y, y')+ s([y, y']))\\
			&=\alpha(y, y')+ s'([y, y'])+\sigma([y, y'])\\
			&=\beta(y, y')+ s'([y, y'])\\
			&= [s'(y), s'(y')]
		\end{split}
	\end{equation}
	$(9), (10)$ and $(11)$ implies that $\phi$ is a Lie algebra homomorphism. Also
	\begin{equation}
		\begin{split}
			((g'\circ \phi) - g)(c_y+s(y))&=(g'\circ \phi)(c_y+s(y))-g_(c_y+ s(y))\\
			&= g'(c_y+s'(y)+\sigma(y))-g(c_y+s(y))\\
			&= 0
		\end{split}
	\end{equation}
	(since, $g'(c_y) = g'(\sigma(y)) = g(c_y)=0$ in $\mathcal{L}(G)$ because $c_y,\sigma(y) \in \mathbb{C}$). Thus $g'\circ \phi = g$. Also $\phi \circ f = f'$. That we got a Lie algebra homomorphism $\phi$ such that the diagram (7) commutes and thus $e$ and $e'$ are equivalent. Therefore, $\Psi$ is one-one.\\
	%\item $\Psi$ is onto.\\
	\par Let $\alpha$ be a 2-cocycle. Define a set $e= \mathcal{L}(G) \oplus \mathbb{C}$ with Lie bracket is given by
	\begin{center}
		$[(x, c), (y, d)] = ([x, y], \alpha(x,y))$
	\end{center}
	Then for any $(x, c), (y, d), (z, k)\in \mathcal{L}(G) \oplus \mathbb{C}$ and $p \in \mathbb{C}$,
	\begin{equation}
		\begin{split}
			[p(x, c)+(y, d), (z, k)]&= [(px+y, pc+d), (z, k)]\\
			&= ([px+y, z], \ \alpha(px+y,z)]\\
			&=(p[x, z]+[y, z], \ p\alpha(x, z)+\alpha(y, z) \\
			&=p[(x,c), (z, k)] + [(y, d), (z, k)]
		\end{split}
	\end{equation}
	Similarly,  $[(x, c), p(y, d)+(z, k)]= [(x, c), (y, d)]+p[(x, c), (z, k)]$ and thus bracket operation is bilinear.\\
	For any $(x, c)\in \mathcal{L}(G)$,
	\begin{center}
		$[(x, c),(x, c)] = ([x, x], \alpha(x, x)) = (0,0)$
	\end{center}
	Also,
	\begin{equation}
		\begin{split}
			[[(x, c),(y, d)], (z,k)]&+[[(y,d),(z,k)],(x,c)]+[[(z,k), (x, c)], (y, d)] \\
			&= [([x,y], \alpha(x,y)), \ (z,k)] +[([y, z],\alpha(y, z)), \  (x, c)] \\
			&\qquad +[([z, x], \alpha(z, x)), \ (y, d)]\\
			&= ([[x, y], z], \ [\alpha(x,y), k]) +([[y, z], x],\ [\alpha(y, z), c])\\
			&\qquad +([[z, x], y], \  [\alpha(z, x), d])\\
			&=(0,0)
		\end{split}
	\end{equation}
	(using tha Jacobi identity of $\mathcal{L}(G)$ and also the fact  that $\mathbb{C}$ is an abelian Lie algebra). Thus $e$ becomes a Lie algebra.\\
	Define
	\begin{center}
		$0\to \mathbb{C} \xrightarrow{f} \mathcal{L}(G)\oplus \mathbb{C} \xrightarrow{g} \mathcal{L}(G) \to 0 $
	\end{center}
	as follows:\\
	\begin{center}
		$f(c)= (0,c)$ and $g(x, c) = x$
	\end{center}
	Then
	\begin{center}
		$f(kc+d) = (0, kc+d) = kf(c)+f(d)$
	\end{center}
	That is, $f$ is linear. Also
	\begin{center}
		$[f(c), f(d)] = [(0,c), (0, d)] = ([0, 0], \alpha(0,0)) =0 $
	\end{center}
	and
	\begin{center}
		$ f([c, d]) = f(0)  = 0$
	\end{center}
	together implies $f([c, d]) = [f(c), f(d)]$. Thus $f$ is a Lie algebra homomorphism. Also
	\begin{equation*}
		\begin{split}
			g(k(x, c)+(y, d))&= g(kx+y, kc+d)\\
			&= kx+y\\
			&= kg(x, c)+g(y, d)
		\end{split}
	\end{equation*}
	and
	\begin{center}
		$[g(x, c), g(y, d)]= [x, y]$ and $g([(x, c), (y, d)]) = [x,y]$
	\end{center}
	implies that $[g(x, c), g(y, d)] = g([(x, c), (y, d)])$. Thus $g$ is a Lie algebra homomorphism.\\
	Also,
	\begin{equation}
		\begin{split}
			Ker(g)&= \{(x,c) \in \mathcal{L}(G)\oplus \mathbb{C} : g(x, c) = 0\}\\
			&=\{(x,c) \in \mathcal{L}(G)\oplus \mathbb{C} : x= 0 \}\\
			&=\{(0, c) : c \in \mathbb{C} \}\\
			&=Im(f)
		\end{split}
	\end{equation}
	and for $(0, c) \in f(\mathbb{C})$  and $(x, d) \in \mathcal{L}(G)\oplus \mathbb{C}$,
	\begin{equation}
		\begin{split}
			[(0,c), (x, d)] = ([0,x], \alpha(0,x)) =(0,0)
		\end{split}
	\end{equation}
	That is, $f(\mathbb{C}) \subseteq Z(\mathcal{L}(G) \oplus \mathbb{C})$,the center of the Lie algebra $\mathcal{L}(G) \oplus \mathbb{C})$. Therefore, $(e; f, g)$ is a central extension of $\mathcal{L}(G)$ by $\mathbb{C}$.\\
	Consider a section $s: \mathcal{L}(G) \to e$ defined by $s(x) = (x, 0)$. Then $(g \circ s)(x) = g(x,0) = x= I_{\mathcal{L}(G)}(x)$ and
	\begin{equation}
		\begin{split}
			[s(x), s(y)]&=[(x,0), (y, 0)]\\
			&=([x, y],[0,0]+\alpha(x,y))\\
			&= ([x, y], 0)+(0, \alpha(x,y))\\
			&= s([x, y])+f(\alpha(x, y))\\
			&= s([x, y])+\alpha(x, y)
		\end{split}
	\end{equation}
	(since, $f$ is an inclusion). That is,
	\begin{center}
		$\alpha(x, y) = [s(x), s(y)] - s([x, y])$
	\end{center}
	That is, $(e; f, g)$ is a central extension with $\Psi([e]) = [\alpha]$. Hence $\Psi$ is onto.

\section{Projective representations of Plesken Lie algebras}
A projective representation of a Plesken Lie algebra $\mathcal{L}(G)$ is a Lie algebra homomorphism $\phi: \mathcal{L}(G) \to \mathfrak{pgl}(V)$ (for some finite dimensional vector space $V$), where $\mathfrak{pgl}(V)$ is the quotient Lie algebra $\mathfrak{gl}(V)/\{kI_V: k \in \mathbb{C} \}$.
\par Every linear representation of $\mathcal{L}(G)$ is a projective representation. For, consider a linear representation $\rho : \mathcal{L}(G) \to \mathfrak{gl}(V)$  of $\mathcal{L}(G)$ and the natural homomorphism $\pi : \mathfrak{gl}(V) \to \mathfrak{pgl}(V)$, then the composition $\pi \circ \rho  : \mathcal{L}(G) \to  \mathfrak{pgl}(V)$ is a projective representation.
\par The following proposition gives a characterization for the projective representation of Lie algebras.

\begin{prop}
	Let  $G$ be a finite group and $\phi$ a projective representation  of $\mathcal{L}(G)$ on $V$. Then there is a linear map $\Phi : \mathcal{L}(G) \to \mathfrak{gl}(V)$ and a bilinear map $\alpha :\mathcal{L}(G)\times \mathcal{L}(G) \to \mathbb{C}$ such that
	\begin{equation}
		[\Phi(x),\Phi(y)]=\alpha(x, y)I_V+\Phi([x, y]) \ \text{ for all }\ x, y \in \mathcal{L}(G)
	\end{equation}
	Conversely, if there is a linear map $\Phi$ and a bilinear map $\alpha$ satisfying $(18)$, then  $\pi \circ \Phi : \mathcal{L}(G) \to \mathfrak{pgl}(V)$  where $\pi : \mathfrak{gl}(V) \to \mathfrak{pgl}(V)$ is the canonical homomorphism, is a projective representation of $\mathcal{L}(G)$. 
\end{prop}
\begin{proof}
	Suppose $\phi : \mathcal{L}(G) \to \mathfrak{pgl}(V)$ is a projective representation and $\pi$ is the natural homomorphism from $\mathfrak{gl}(V)$ to $\mathfrak{pgl}(V)$. Let $X$ be the coset representatives of $\mathfrak{gl}(V)$ in $\mathfrak{pgl}(V)$. Define $\Phi : \mathcal{L}(G) \to \mathfrak{gl}(V)$ by choosing for each $x \in \mathcal{L}(G)$, an element $A_x$ in $\mathfrak{gl}(V)$ such that $\pi(A_x) = \phi(x)$. Then ,
	\begin{center}
		$\Phi(x) = A_x$
	\end{center}
	is a linear map. Also,
	\begin{equation}
		\begin{split}
			[\Phi(x),\Phi(y)] - \Phi([x,y])&= [A_x, A_y]  - A_{[x, y]} \\
			&= \alpha(x, y)I_V
		\end{split}
	\end{equation}
	(since, $\pi([A_x, A_y]  - A_{[x, y]})= [\phi(x), \phi(y)] - \phi([x,y])= 0$ in $\mathfrak{pgl}(V)$) where $\alpha : \mathcal{L}(G) \times \mathcal{L}(G) \to \mathbb{C}$ is a bilinear map. Conversely given a linear map $\Phi : \mathcal{L}(G) \to \mathfrak{gl}(V)$ satisfying $(18)$.
	\par Consider
	\begin{center}
		$\mathcal{L}(G) \xrightarrow{\Phi} \mathfrak{gl}(V) \xrightarrow{\pi} \mathfrak{pgl}(V)$
	\end{center}
	where $\pi$ is the natural homomorphism.	Then
	\begin{equation}
		\begin{split}
			[(\pi \circ \Phi)(x), (\pi \circ \Phi)(y)]&= [\pi(\Phi(x)), \pi(\Phi(y))] \\
			&=\pi([\Phi(x), \Phi(y)])\\
			&=\pi(\alpha(x,y)I_V + \Phi([x,y])) \\
			&=\pi(\Phi([x,y])) = (\pi \circ \Phi)([x, y])
		\end{split}
	\end{equation}
	That is, $\pi \circ \Phi$ is a projective representation of $\mathcal{L}(G)$.
\end{proof}

Now we can define the projective representation of $\mathcal{L}(G)$ as follows:
\begin{definition}
	A linear map $\Phi: \mathcal{L}(G) \to \mathfrak{gl}(V)$ is a projective representation of $\mathcal{L}(G)$ if there exists a bilinear map $\alpha :\mathcal{L}(G) \times \mathcal{L}(G) \to \mathbb{C}$ such that
	\begin{center}
		$[\Phi(x),\Phi(y)]= \alpha(x, y)I_V + \Phi([x,y])$
	\end{center}
	
\end{definition}
Recall that an extension $0 \to \mathcal{L}(H)\xrightarrow{f} e \xrightarrow{g} \mathcal{L}(G) \to 0$ \textit{splits} if there exist a Lie algebra homomorphism $s : e \to \mathcal{L}(G)$ such that $g \circ s = 1_{\mathcal{L}(G)}$. It is also known that an extension splits if and only if $e$ is isomorphic to the semidirect product Lie algebra $\mathcal{L}(H) \rtimes \mathcal{L}(G)$ (see \cite{weibel}). 
\begin{example}
	Consider the central extension $0 \to \mathcal{L}(H)\xrightarrow{f} e \xrightarrow{g} \mathcal{L}(G) \to 0$ in Example \ref{central}. This is not a split extension since $e \ncong \mathcal{L}(H) \rtimes \mathcal{L}(G)$. Define $\alpha : \mathcal{L}(G) \times \mathcal{L}(G) \to \mathcal{L}(H)$ by 
	\begin{center}
		$\alpha(x, y) = [s(x), s(y)] - s([x, y])$
	\end{center}
	and $\Phi : \mathcal{L}(G) \to \mathfrak{gl}(V)$ by
	\begin{center}
		$\Phi(x) = s(x) I_V$
	\end{center}
	Then $\alpha$ is a bilinear map and $\Phi$ is a linear map such that
	\begin{center}
		$[\Phi(x), \Phi(y)] = \alpha(x, y)I_V + \Phi([x, y])$.
	\end{center}
\end{example}

Consider the projective representation $\Phi: \mathcal{L}(G)\to \mathfrak{gl}(V)$, then for any $\Phi(x),\Phi(y), \Phi(z) \in \mathfrak{gl}(V)$,
\begin{equation}
	[[\Phi(x),\Phi(y)],\Phi(z)]+ [[\Phi(y),\Phi(z)],\Phi(x)]+ [[\Phi(z),\Phi(x)],\Phi(y)]= 0
\end{equation}
(by Jacobi identity in $\mathfrak{gl}(V)$). Also 
\begin{equation}
	\begin{split}
		[[\Phi(x),\Phi(y)],\Phi(z)]&= [\alpha(x,y)I_V+\Phi([x, y]), \Phi(z)]\\
		&= \alpha(x, y)[I_V, \Phi(z)]+ [\Phi([x, y]), \Phi(z)]\\
		&= [\Phi([x, y]), \Phi(z)]\\
		&= \alpha([x,y],z)I_V+ \Phi([[x,y],z])
	\end{split}
\end{equation}
and 
\begin{equation}
	\begin{split}
		[[\Phi(x),\Phi(y)],\Phi(z)]+& [[\Phi(y),\Phi(z)],\Phi(x)]+ [[\Phi(z),\Phi(x)],\Phi(y)]\\
		&= \alpha([x,y],z)I_V+ \Phi([[x,y],z])+\alpha([y,z],x)I_V\\
		&\quad + \Phi([[y,z],x])+\alpha([z,x],y)I_V+ \Phi([[z,x],y]) \\
		&=(\alpha([x,y],z)+\alpha([y,z],x)+\alpha([z,x],y))I_V\\
		&\quad +\Phi([[x,y],z]+[[y,z],x]+[[z,x],y])\\
		&=(\alpha([x,y],z)+\alpha([y,z],x)+\alpha([z,x],y))I_V
	\end{split}
\end{equation}
Then from $(23)$,
\begin{equation}
	\alpha([x,y],z)+\alpha([y,z],x)+\alpha([z,x],y)=0
\end{equation}
$\alpha$ is a bilinear map and satisfies the 2-cocycle condition. That is, $\alpha \in Z^2(\mathcal{L}(G),\mathbb{C})$.  Thus the projective representation $\Phi$ is also referred to as an \textit{$\alpha$-representation } on the vector space $V$.

\subsection{Projectively equivalent $\alpha$-representations of Plesken Lie algebras}
Certain equivalences on projective representations $\mathcal{L}(G)$ of Plesken Lie algebra  of a finite group $G$ defined below turnsout to be significant in the study of representations of Plesken Lie algebras.

\begin{definition}
	Let $\Phi_1$ be an $\alpha_1$-representation and $\Phi_2$ be an $\alpha_2$-representation of $\mathcal{L}(G)$. Then $\Phi_1$ and $\Phi_2$ are projectively equivalent if there exists an isomorphism $f : V \to W$ and a linear map $\delta: \mathcal{L}(G) \to \mathbb{C}$ such that
	\begin{center}
		$\Phi_2(x) = f\circ \Phi_1(x)\circ f^{-1} + \delta(x)I_W$ for all $x \in \mathcal{L}(G)$
	\end{center}
	
	If $\delta(x)= 0$ for all $x \in \mathcal{L}(G)$, then $\Phi_1$ and $\Phi_2$ are linearly equivalent.
\end{definition}

\begin{example}
	Let $\Phi_1$ and $\Phi_2$ be two projective representations of $\mathcal{L}(D_8)$ (where $D_8= <a,b : a_4 = b^2= e, bab = a^{-1} >$) where : \\
	$\Phi_1: \mathcal{L}(D_8) \to \mathfrak{gl}(\mathbb{C}^2)$ is defined by
	\begin{center}
		$\Phi_1(c\hat{a}) = c\begin{pmatrix}0&2\\-2&0\end{pmatrix}$
	\end{center}
	and $\Phi_2: \mathcal{L}(D_8) \to \mathfrak{gl}(\mathbb{C}^2)$ is defined by
	\begin{center}
		$\Phi_2(c\hat{a}) = c \begin{pmatrix}2i&0\\0&-2i\end{pmatrix}$
	\end{center}
	Let $f :\mathbb{C}^2\to \mathbb{C}^2$ be
	\begin{center}
		$f(v_1,v_2) = \begin{pmatrix}\frac{v_1-iv_2}{2}, \frac{v_1+iv_2}{2}\end{pmatrix}$
	\end{center}
	and $\delta  : \mathcal{L}(D_8) \to \mathbb{C}$ be $\delta(x) = 0$ for all $x\in \mathcal{L}(D_8)$. It is easy to verify that
	\begin{center}
		$ \Phi_2(x)(v_1,v_2) = ( f \circ \Phi_1(x) \circ f^{-1})(v_1,v_2) + \delta(x)I_{\mathbb{C}^2}(v_1, v_2)$ for all $x \in \mathcal{L}(D_8)$
	\end{center}
	thus $\Phi_1$ and $\Phi_2$ are linearly equivalent.
\end{example}
Next we proceed to establish the correspondence between  two projectively equivalent projective representations of $\mathcal{L}(G)$ and $H^2(\mathcal{L}(G), \mathbb{C})$. 
\begin{lemma}
	Let $\Phi_1$ be an $\alpha_1$-representation and $\Phi_2$ be an $\alpha_2$-representation of $\mathcal{L}(G)$. If $\Phi_1$ is projectively equivalent to $\Phi_2$, then $\alpha_1$ and $\alpha_2$ are cohomologous.
\end{lemma}
\begin{proof}
	Suppose $\Phi_1$ and $\Phi_2$ are projectively equivalent projective representations of $\mathcal{L}(G)$. Then by definition, there exists an isomorphism $f : V \to W$ and a linear map $\delta: \mathcal{L}(G) \to \mathbb{C}$ such that
	\begin{center}
		$\Phi_2(x)(v) = f \circ \Phi_2(x) \circ f^{-1}+\delta(x)I_W$ for all $x \in \mathcal{L}(G)$
	\end{center}
	Then we have,
	\begin{equation*}
		\begin{split}
			\alpha_2(x,y)I_W&=[\Phi_2(x),\Phi_2(y)] - \Phi_2([x,y])\\
			&=[f \circ \Phi_2(x) \circ f^{-1}+\delta(x)I_W, f \circ \Phi_2(y) \circ f^{-1}+\delta(x)I_W] \\
			&\hspace{1in}- (f \circ \Phi_2([x,y]) \circ f^{-1}+\delta([x, y])I_W)\\
			&=[f \circ \Phi_2(x) \circ f^{-1}, f \circ \Phi_2(y) \circ f^{-1}] - f \circ \Phi_2([x, y]) \circ f^{-1} - \delta([x, y])I_W\\
			&=f\circ ([\Phi_1(x), \Phi_1(y)]- \Phi_1([x, y]))\circ f^{-1} - \delta([x, y])I_W\\
			&= f\circ \alpha_1(x, y)I_V\circ f^{-1} - \delta([x, y])I_W\\
			&=\alpha_1(x, y)I_W - \delta([x, y])I_W\\
		\end{split}
	\end{equation*}
	Thus $\alpha_1$ and $\alpha_2$ are cohomologous.
\end{proof}

\begin{lemma}
	Let $\Phi_1$ be an $\alpha_1$-representation of $\mathcal{L}(G)$ and $\alpha_2$ is a 2-cocycle cohomologous to $\alpha_1$, then there exists an $\alpha _2$-representation $\Phi_2$ of $\mathcal{L}(G)$ which  is projectively equivalent to $\Phi_1$.
\end{lemma}

\begin{proof}
	Suppose $\Phi_1$ is an $\alpha_1$-representation of $\mathcal{L}(G)$ and $\alpha_2$ is a 2-cocycle which is cohomologous to $\alpha_1$. Then there exists a linear map $\sigma : \mathcal{L}(G) \to \mathbb{C}$ such that
	\begin{equation}
		\alpha_1(x, y) - \alpha_2(x, y) = -\sigma([x, y])
	\end{equation}
	define $\Phi_2 : \mathcal{L}(G) \to \mathfrak{gl}(V)$ by
	\begin{center}
		$\Phi_2(x)  = \Phi_1(x) - \sigma(x)I_V$.
	\end{center}
	Clearly $\Phi_2$ is a linear map and
	\begin{equation}
		\begin{split}
			[\Phi_2(x), \Phi_2(y)]&= [\Phi_1(x) - \sigma(x)I_v, \Phi_1(y) - \sigma(y)I_v] \\
			&= [\Phi_1(x), \Phi_1(y)]\\
			&= \alpha_1(x, y)I_V + \Phi_1([x, y])\\
			&= \alpha_1(x, y)I_V+ \Phi_2([x, y]) + \sigma([x, y])I_V\\
			&=\alpha_2([x, y])I_V +  \Phi_2([x, y])
		\end{split}
	\end{equation}
	ie., $\Phi_2$ is an $\alpha_2$-representation projectively equivalent to $\Phi_1$ (take $f = I_V$ ).
\end{proof}
Thus, any projective representation up to projective equivalence defines an element of the second cohomology group $H^2(\mathcal{L}(G), \mathbb{C})$.

\section{Conclusion}
\begin{itemize}
	\item The classical First and Second Whitehead Lemmas state that the first, respectively second, cohomology group of a finite-dimensional semisimple Lie algebra with coefficients in any finite-dimensional module vanishes (see \cite{whitehead}). Thus from section 3, we can conclude that for semisimple Plesken Lie algebras, the projective representations are precisely the linear representations. Also we can conclude that the central extensions of a semisimple Plesken Lie algebra $\mathcal{L}(G)$ by $\mathbb{C}$ are equivalent and it is the Lie algebra $\mathcal{L}(G) \oplus \mathbb{C}$ with the Lie bracket $[(x, c), (y,d)] = ([x,y], [c, d])$ for $(x, c), (y,  d) \in \mathcal{L}(G) \oplus \mathbb{C}$.
	\item We have there is a bijective correspondence between set of equivalence class of central extensions of $\mathcal{L}(G)$ and $\mathbb{C}$ and $H^2(\mathcal{L}(G), \mathbb{C})$. Also there is a bijective correspondence between set of projective equivalence class of projective representations of $\mathcal{L}(G)$ and  $H^2(\mathcal{L}(G), \mathbb{C})$. Thus we can classify the projectively equivalent projective representations of $\mathcal{L}(G)$ by central extensions of $\mathcal{L}(G)$ by $\mathbb{C}$.
\end{itemize}

\end{proof}


\begin{thebibliography}{99}
	\bibitem{romeo}{S. N. Arjun and P. G. Romeo}, ``On Representations of Plesken Lie algebras", Asian- European Journal of Mathematics, 2022.
	
	\bibitem{Batten}{P. G. Batten},  ``Multilpliers and covers of Lie algebras'', Ph.D Thesis, 1993.
	
	\bibitem{Plesken}  {Arjeh M. Cohen  and D. E. Taylor}, ``On a Certain Lie algebra Defined by a Finite Group'',
	The American Mathematical Monthly, Aug. - Sep., 2007.
	
	\bibitem{Pleskenfinite}{ John Cullinan and Mona Merling}, ``On the Plesken Lie algebra defined over finite field''.
	
	\bibitem{Fulton} {William Fulton and Joe Harris}, ``Representation theory, A first course'', GTM-129, Springer.
	
	\bibitem{Karpilovsky}{ Gregory Karpilovsky},`` Projective representations of finite groups``, New York-Basel, 1985.
	
	\bibitem{Hall} {Brian C. Hall}, ``Introduction to Lie Algebras and Representation Theory'', Springer-Verlag, 1972.
	
	\bibitem{Sumana}{Sumana Hatui, E. K. Narayanan and Pooja Singla}, ``On projective representations of finitely generated groups'',  \textit{arxiv.2006.02832v2[math.RT]}, 2020.
	
	\bibitem{Sumana}{Sumana Hatui and Pooja Singla}, ``On Schur multiplier and projective representations of Heisenberg groups'', Journal of Pure and Applied Algebra, MArch 2021.
	
	\bibitem{Humphreys} {James E. Humphreys}, ``Introduction to Lie Algebras and Representation Theory'', Springer - Verlag, 1972.
	
	\bibitem{James} {Gordon James and Martin Liebeck}, ''Representations and Characters of Groups'',  Second Edition, Cambridge University Press, 2001.
	
	\bibitem{mendonca} Eduardo Monteiro Mendonca, "Projective representations of groups", 2017.
	
	\bibitem{adfunctor} {Alex Turzillo}, ''Representations of Matrix Lie algebras'', REU Apprentice Program, University of Chicago, August 2010.
	
	\bibitem{weibel} {Charles A. Weibel}, ''An introduction to homological algebra", Cambridge University Press, 1994.
	
	
	\bibitem{whitehead} Pasha Zusmanovich, ``A Converse to the Whitehead Lemma",  https://arxiv.org/abs/0704.3864.
	
	
\end{thebibliography}
\end{document}